\documentclass[a4paper,11pt]{amsart}
\usepackage[english]{babel}
\usepackage[centertags]{amsmath}
\usepackage{latexsym, amsfonts, amssymb, graphicx}
\usepackage{hyperref}
\usepackage[all,dvips]{xy}
\usepackage{graphicx}
\usepackage[latin1]{inputenc}
\usepackage{url}
\usepackage{enumerate}
\usepackage{multirow}
\usepackage{fourier}

\theoremstyle{plain}
\newtheorem{thm}{Theorem}[section]
\newtheorem{lem}[thm]{Lemma}
\newtheorem{prop}[thm]{Proposition}
\newtheorem{cor}[thm]{Corollary}

\theoremstyle{definition}
\newtheorem{defn}[thm]{Definition}
\newtheorem{nota}[thm]{Notation}

\theoremstyle{remark}
\newtheorem{rem}[thm]{Remark}

\newtheorem{exmp}[thm]{Example}

\def\N{{\mathbf N}}

\def\P{{\mathbf P}}
\def\C{{\mathcal C}}

\def\O{\mathcal{O}}

\newcommand{\F}{\mathbf{F}}
\newcommand{\res}{\textrm{res}}

\makeatletter
\def\bign#1{\mathclose{\hbox{$\left#1\vbox to8.5\p@{}\right.\n@space$}}\mathopen{}}
\def\Bign#1{\mathclose{\hbox{$\left#1\vbox to11.5\p@{}\right.\n@space$}}\mathopen{}}
\def\biggn#1{\mathclose{\hbox{$\left#1\vbox to14.5\p@{}\right.\n@space$}}\mathopen{}}
\def\Biggn#1{\mathclose{\hbox{$\left#1\vbox to17.5\p@{}\right.\n@space$}}\mathopen{}}
\makeatother

\begin{document}

\title{Codes and the Cartier Operator}

\author{Alain Couvreur}

\address{INRIA Saclay \^Ile-de-France --- CNRS LIX, UMR 7161, \'Ecole Polytechnique, 91128 Palaiseau Cedex}
\email{alain.couvreur@lix.polytechnique.fr}

%\date{\today}

\thispagestyle{empty}

\begin{abstract}
In this article, we present a new construction of codes from algebraic curves.
Given a curve over a non-prime finite field, the obtained codes are defined over a subfield.
We call them {\it Cartier Codes} since their construction involves the Cartier operator.
This new class of codes can be regarded as a natural geometric generalisation of classical Goppa codes.  
In particular, we prove that a well-known property satisfied by classical Goppa codes extends naturally to Cartier codes.
We prove general lower bounds for the dimension and the minimum distance of these codes and
compare our construction with a classical one: the subfield subcodes of Algebraic Geometry codes.
We prove that every Cartier code
is contained in a subfield subcode of an Algebraic Geometry code and that the two
constructions have similar asymptotic performances.

We also show that some known results on subfield subcodes of Algebraic Geometry codes can be proved nicely by using properties of the Cartier operator and that some known bounds on the dimension of subfield subcodes of Algebraic Geometry codes can be improved thanks to Cartier codes and the Cartier operator.
\end{abstract}

\maketitle


\section*{Introduction}
It is well-known that, with a high probability, a random code is good. However, getting explicit asymptotically good families of codes is not a simple task. 
In the beginning of the eighties Tsfasman, Vl\u{a}du\c{t} and Zink \cite{TVZ} and independently Ihara \cite{Ihara}, proved the existence of asymptotically good infinite families of Algebraic Geometry codes (AG codes) over $\F_q$ for all $q\geq 5$. They proved in particular that for a square $q \geq 49$, some families of AG codes over $\F_q$ beat the Gilbert--Varshamov bound.

For small values of $q$ and in particular for $q=2$, the use of AG codes does not seem to be suitable to produce asymptotically good families of codes. 
A classical approach to construct good codes over small fields is to construct good codes over a finite extension and then use a ``descent'' operation such as \emph{trace codes} or \emph{subfield subcodes} (see \cite[Chapter 9]{sti}).
This is for instance the point of BCH codes, classical Goppa codes or more generally of alternant codes.
For this reason, studying subfield subcodes of AG code is natural.

As far as we know, the first contributions on subfield subcodes of AG codes are due to Katsman and Tsfasman \cite{KatsmanTsfasman} and independently to Wirtz \cite{Wirtz}. Both obtained lower bounds for the dimension of such codes exceeded the generic formulas for subfield subcodes. Upper bounds on the covering radius and the minimum distance of such codes are proved by Skorobogatov in \cite{Skoro}.
Subsequently, Stichtenoth showed in \cite{Sticht_Subfield} that the lower bounds for the dimension due to Katsman et al. and Wirtz are the consequence of a general result on subfield subcodes. 
%Improvements of the estimate of the dimension have been given by Shibuya et al. for codes from $C_{ab}$ curves. 

This article presents a new construction of codes on a finite field $\F_q$ from a curve over an extension $\F_{q^\ell}$.
Our method differs from that of the above-cited references since it is not based on the use of the subfield subcode operation.
The key point of our method is to use differentials fixed by Cartier, that is logarithmic differentials. This is the reason why, we call these codes {\it Cartier codes} and denote them by $Car_q (D,G)$.
These {\it Cartier codes} can be regarded as a natural generalisation of classical Goppa codes, which turn out to be Cartier codes from a curve of genus $0$.
Moreover, it is well-known that given a squarefree polynomial $f$, the classical Goppa codes associated to $f^{q-1}$ and $f^q$ are equal and, as we show in Theorem \ref{EgaliteGen}, this property extends naturally to Cartier codes.

We then study the relations between Cartier codes and the subfield subcodes of AG codes. We prove that a Cartier code is a always a subcode of such a subfield subcode and prove Theorem \ref{thm:codim} yielding an upper bound for the dimension of the corresponding quotient space.
We discuss the minimum distance of Cartier codes and prove two lower bounds for their dimensions in Theorems \ref{dimC} and \ref{thm:improve_dim}.
Finally, thanks to Cartier codes and the Cartier operator, we improve in Corollary \ref{cor:impr} the known estimates for the dimension of subfield subcodes of AG codes $C_{\Omega} (D, G)_{|\F_q}$ when $G$ is non-positive. We also observe that Cartier codes have similar asymptotic performances as subfield subcodes of AG codes.

Thanks to our approach involving the Cartier operator, we are able to give new proofs of some results on subfield subcodes of AG codes. Such new proofs, which are clearly less technical than the original ones are presented in \S \ref{sec:motiv} and Remark \ref{rem:alter_proof}. 
As far as we know, the Cartier operator has never been used in algebraic geometric coding theory up to now.

This article is organised as follows.  Basic notions on subfield subcodes and classical Goppa codes are recalled in Section \ref{sec:pre}. Section \ref{SecAG} is a brief review on AG codes and known results on their subfield subcodes.
After recalling the definition together with some basic features of the Cartier operator, we prove a vanishing property of this map in Section \ref{SecCartier}.
In Section \ref{secCar}, we introduce Cartier codes. We compare them with subfield subcodes of AG codes in Section \ref{sec:compare}. In Section \ref{sec:param}, we discuss the parameters and the asymptotic performances of Cartier codes and subfield subcodes of AG codes. Section \ref{sec:examples} is devoted to examples of Cartier codes from the Klein quartic which illustrate the previous results.

\section{Preliminaries: Classical Goppa codes}\label{sec:pre}

\begin{nota}
Consider a finite extension $\F_{q^\ell}/\F_q$ of finite fields and let $C$ be a code of length $n$ over $\F_{q^\ell}$.
The subfield subcode $C\cap \F_q^n$ over $\F_q$ is denoted by $C_{|\F_q}$.   
\end{nota}

% Recall the following well-known and elementary result.

\begin{lem}\label{SubParam}
  Let $C$ be a code over $\F_{q^\ell}$ with parameters $[n,n-r,d]_{q^\ell}$, then $C_{|\F_q}$ has parameters $[n, \geq n-\ell r, \geq d ]_q$. 
\end{lem}

\begin{proof}
  \cite[Lemma 9.1.3]{sti}
\end{proof}

%Classical Goppa codes is a subfamily of the family of alternant codes (\cite[Chapter 12]{slmc1}).
%An alternant code $C$ is a code of the form $C'\cap \F_q^n$, where $C'$ is a GRS code of length $n$ defined over some finite extension $\F_{q^\ell}$ of $\F_q$.

\begin{defn}[Classical Goppa codes]\label{goppa}
  Let $L:=(\alpha_1, \ldots , \alpha_n)$ be an ordered $n$--tuple of distinct elements of a field $\F_{q^\ell}$. Let $f\in \F_{q^\ell}[x]$ be a polynomial which does not vanish at any of the elements of $L$.
The Goppa code $\Gamma_q (L,f)$ is defined by
$$
\Gamma_{q} (L,f):= \left\{ (c_1, \ldots , c_n)\in \F_{q}^n\ \left|  \ \sum_{i=1}^n \frac{c_i}{x-\alpha_i}\equiv 0 \mod (f)\right. \right\} \cdot
$$
\end{defn}

One can prove that $\Gamma_q (L, f)$ is alternant, i.e. is a subfield subcode of a Generalised Reed--Solomon code over $\F_{q^\ell}$ (\cite[Theorem 12.3.4]{slmc1}) with parameters $[n, n-\deg(f), \deg(f)+1]_{q^\ell}$. From Lemma \ref{SubParam}, the code $\Gamma_q(L,f)$ has parameters 
$$[n, \geq n-\ell \deg (f), \geq \deg (f)+1]_q.$$
These estimates can be improved in some situations thanks to the following well--known result.

\begin{thm}\label{Thm:Egalite}
  Let $L$ and $f$ be as in Definition \ref{goppa}. If $f$ is squarefree. Then,
$$
\Gamma_q (L, f^{q-1})=\Gamma_q (L,f^q).
$$
Thus, the parameters of this code satisfy $[n, \geq n-\ell (q-1)\deg(f), \geq q\deg(f)+1]_q$.
\end{thm}

\begin{proof}
  See \cite{Some_Japanese} or \cite[Theorem 4.1]{BernLangPeters}.
\end{proof}

Another proof of Theorem \ref{Thm:Egalite} involving the Cartier operator is given in \S~\ref{sec:motiv}.

\begin{rem}
A more general version of the statement could be, let $f_1, \ldots , f_s$ and $h_1, \ldots ,$ $h_t$ be irreducible polynomials in $\F_{q^{\ell}}[x]$, let $a_1, \ldots, a_s, b_1, \ldots , b_t$ be positive integers such that for all $i$, $a_i \equiv q-1 \mod q$ and for all $j$, $b_j \not \equiv q-1 \mod q$, then
$$
\Gamma_q (L, f_1^{a_1}\cdots f_s^{a_s}h_1^{b_1}\ldots h_t^{b_t}) = \Gamma_q (L, f_1^{a_1+1}\cdots f_s^{a_s+1} h_1^{b_1}\ldots h_t^{b_t}).
$$  
\end{rem}

\section{Algebraic Geometry codes}\label{SecAG}

\subsection{Caution} Since AG codes have been introduced by Goppa in \cite{goppa}, they are frequently referred as {\it Goppa codes} or {\it geometric Goppa codes}. However, AG codes are not a generalisation of classical Goppa codes since they do not involve the subfield subcode operation. Actually, AG codes are a generalisation of Reed--Solomon Codes.

\subsection{Context, notation and prerequisites}\label{Subsec:context}
In this article, a curve is smooth projective and geometrically irreducible.
Given a curve $X$ over a field $\F$. We denote by $F$ its function field, by $\Omega_{F/\F}$ its space of rational differential forms and by $g$ its genus.
Given a place $P$ of $F$ , we denote respectively by $\mathcal{O}_{X,P}$ , $\mathfrak{m}_{X,P}$ , $\F_q (P)$ and $v_P$ the local
ring at $P$ , its maximal ideal, the residue field and the valuation at $P$. 
For a divisor $A$ on $X$, we denote by $L(A)$ the space $L(A):=H^0(X, \mathcal{O}_X(A))$ and by $\Omega(A)$ the space
$
\Omega (A):= H^0 (X, \Omega_{X}\otimes \mathcal{O}_X(-A)).$
The $\F_q$--dimen\-sions of these spaces are respectively denoted by $h^0(A)$ and $h^1(A)$.

\begin{defn}\label{defn:CodeGeo}
Let $X$ be a curve over $\F_q$, let $G$ be a divisor on $X$ and $P_1, \ldots , P_n$ be distinct rational points of $X$ avoiding the support of $G$. Set $D:=P_1+\cdots +P_n$.
The code $C_{\Omega} (D,G)$ is defined as the image of the map
$$
\res_D: \left\{
  \begin{array}{ccc}
    \Omega (G-D) & \longrightarrow & \F_q^n \\
    \omega & \longmapsto & (\res_{P_1}(\omega), \ldots , \res_{P_n}(\omega))
  \end{array}
\right. .
$$
\end{defn}

\noindent If $\deg (G)>2g-2$ (or at least if $h^1 (G)=0$), then the dimension and the minimum distance  of such a code satisfy the following well--known lower bounds.
\begin{eqnarray}
 \label{eq:dim_const}  \dim (C_{\Omega}(D,G)) & \geq & n-(\deg (G)+1 -g) \\
 \label{eq:dist_const} d(C_{\Omega} (D,G)) & \geq & \deg (G)+2-2g.
\end{eqnarray}

\begin{exmp}\label{Goppa=diff}
With the notations of Definition \ref{goppa}, let $X$ be the projective line over $\F_q$, let $P_1, \ldots , P_n$ be the points of coordinates $((\alpha_1:1), \ldots , (\alpha_n:1))$ and $P$ be the point $(1:0)$. Regarding $f$ as a rational function on $\P^1$,
let $E:=(f)_0$ be the divisor of the zeroes of $f$. 
The Goppa code $\Gamma_q (L,f)$ is nothing but
$
C_{\Omega} (D,E-P)_{|\F_q}
$ (see \cite[Example 3.4]{hohpel}).
\end{exmp}

To conclude these prerequisites, we recall the following definition which is useful in what follows.

\begin{defn}
  A positive divisor $G$ on a curve $X$ is said to be \emph{reduced} if the corresponding subscheme of $X$ is reduced. That is, $G$ is a formal sum of places $G=m_1 Q_1+\cdots +m_s Q_s$, where all the $m_i$'s are equal to $1$. 
\end{defn}

\subsection{Subfield subcodes of Algebraic Geometry codes}
Given a curve $X$ on $\F_{q^{\ell}}$ with $\F_{q^{\ell}}$--divisors $D, G$ as in Definition \ref{defn:CodeGeo}. Let us consider the code $C_{\Omega} (D, G)_{|\F_q}$.
Lemma \ref{SubParam} together with (\ref{eq:dim_const}) and (\ref{eq:dist_const}) assert that if $h^1 (G)=0$, then this code over $\F_q$ has parameters
$$[n,k \geq n-\ell(\deg (G)+1-g),d \geq \deg (G)+2-2g].$$

The lower bound for the dimension has been improved by Kastman and Tsfasman \cite{KatsmanTsfasman}  and independently by Wirtz \cite{Wirtz} under some condition on $G$. Namely, if $\deg (G)>2g-2$ and $G\geq qG_1 \geq 0$ for some divisor $G_1$, then from \cite[Theorem 1]{Wirtz},
\begin{equation}
  \label{eq:wirtz_dimension}
k \geq n-1-\ell \deg (G-G_1)+\ell h^1 (G_1).
\end{equation}

\medbreak

\noindent Subsequently, Stichtenoth proved the following generalisation of (\ref{eq:wirtz_dimension}).

\begin{thm}[{\cite[Theorem 4]{Sticht_Subfield}}]\label{thm:dim_sticht}
Let $X$ be a curve over $\F_{q^{\ell}}$. Let $G, D$ be as in Definition \ref{defn:CodeGeo} and $G\geq qG_1$ (possibly non positive), then 
$$
\dim (C_{\Omega} (D,G)_{|\F_q}) \geq \left\{
  \begin{array}{lcc}
    n-1-\ell (h^0(G)-h^0(G_1)) & {\rm if} & G\geq 0\\
    n-\ell (h^0 (G) - h^0 (G_1)) & {\rm if} & G\ngeq 0
  \end{array}
\right. .
$$
\end{thm}

\begin{rem}
The original statement \cite[Theorem 4]{Sticht_Subfield} requires the hypothesis $\deg (G)<n$. This hypothesis is actually useless. Indeed, \cite[Theorem 1]{Sticht_Subfield} gives
$$
\dim (C_{\Omega} (D,G)_{|\F_q}) \geq \left\{
  \begin{array}{lcc}
    n-1-\ell (h^0(G)-h^0 (G-D)-h^0(G_1)+h^0(G_1-D))  & {\rm if} & G\geq 0\\
    n-\ell (h^0 (G) -h^0 (G-D) - h^0 (G_1)+h^0(G_1-D)) & {\rm if} & G\ngeq 0
  \end{array}
\right. ,
$$
which is at least as good as Theorem \ref{thm:dim_sticht}.
\end{rem}

\begin{rem}
For $h^1(G)=0$ and $G\geq 0$, Theorem \ref{thm:dim_sticht} gives exactly (\ref{eq:wirtz_dimension}).  
\end{rem}

On the other hand, Wirtz proposed also a generalisation of Theorem \ref{Thm:Egalite}.

\begin{thm}[{\cite[Theorem 2]{Wirtz}}]\label{Thm:Wirtz_2}
Let $X, D, G, G_1$ be as in Theorem \ref{thm:dim_sticht}. Assume that $\deg (G_1)\geq 2g-2$ and $G_1\geq 0$. Let $G_U$ be the reduced divisor defined as the sum of the places $P$ such that $v_P(G)\equiv q-1 \mod q$. Then,
$$
C_{\Omega}(D, G)_{|\F_q} = C_{\Omega} (D, G+G_U)_{|\F_q}.
$$
\end{thm}

\section{The Cartier operator}\label{SecCartier}

The Cartier operator is a semi-linear endomorphism of the space of rational differential forms in positive characteristic. In terms of Serre duality, it corresponds to the adjoint of the Frobenius map.

%\subsection{Context}
We keep the context of Section \ref{SecAG}. Moreover, in what follows, $x$ denotes a separating element of $F/\F_q$ (\cite[\S 3.10]{sti}). We denote by $p$ the characteristic of $F$ and by $\F_p$ the corresponding prime field.

%\subsection{Definition}
\begin{defn}[The Cartier operator]
  Let $\omega \in \Omega_{F/\F_q}$. There exists $f_0, \ldots , f_{p-1}$ such that
$
\omega=(f_0^p+f_1^p x+\cdots +f_{p-1}^px^{p-1})dx.
$
The Cartier operator $\C$ is defined by
$$
\C(\omega):=f_{p-1} dx.
$$
The definition does not depend on the choice of $x$ (see \cite[Proposition 1]{Seshadri}).
\end{defn}

\subsection{Local and global Properties of the Cartier operator}

%We summarise some well-known properties of the Cartier operator.
We refer the reader to \cite{car, Cartier, Seshadri, TVN} for the proofs of the following statements.

\begin{prop}[Local properties of $\C$]\label{local}
  Let $P$ be a place of $F$. For all $\omega \in \Omega_{F/\F_{q^{\ell}}}$,
  \begin{enumerate}[(i)]
  \item\label{posit} $v_P(\omega)\geq 0\ \Longrightarrow \ v_P (\C(\omega)) \geq 0$;
  \item $v_P(\omega)\leq -2\ \Longrightarrow \ v_P (\C(\omega))>v_P (\omega)$;
  \item\label{val-1} $v_P(\omega)=-1\ \Longrightarrow \ v_P (\C(\omega))=-1$;
  \item\label{resi} $\res_P(\C(\omega))=\res_P(\omega)^{1/p}$.

  \end{enumerate}
\end{prop}

\begin{prop}[Global Properties of $\C$]\label{glob}
  For all $\omega \in \Omega_{F/\F_q}$ and all $f\in F$,
  \begin{enumerate}[(i)]
  \item\label{lin} $\C(f^p \omega)=f \C(\omega)$;
  \item $\C(\omega)=0 \ \Longleftrightarrow \ \exists h\in F,\ \omega=dh$;
  \item\label{log} $\C(\omega)=\omega \ \Longleftrightarrow \ \exists h\in F,\ \omega=\frac{dh}{h}$.
  \end{enumerate}
\end{prop} 

% \begin{rem}
%   Item (\ref{lin}) in the above proposition entails that $\C$ is $\F_p$--linear but not $\F_q$--linear (it is only $\F_q$--semi-linear). If $q=p^\ell$, then $\C^\ell$ is $\F_q$--linear.
% \end{rem}

% \begin{prop}\label{prop:Valuationq}
% Let $\omega \in \Omega_{F/\F_q}$ and $P$ be a place of $F$.
% Then .  
% \end{prop}

% \begin{proof}
% \cite[Exercise 2.2.53]{TVN}.
% \end{proof}

\begin{nota}
  From now on, for $q=p^\ell$, we denote by $\C_q$ the $\ell$ times iterated Cartier operator $\C_q:=\C^\ell$. This map is $\F_q$--linear.
Replacing $p$ by $q$, Propositions \ref{local}(\ref{posit}--\ref{resi}) and \ref{glob}(\ref{lin}) extend naturally to $\C_q$.
\end{nota}

\begin{cor}\label{cor:added}
 Let $\omega \in \Omega_{F/\F_q}$ and $P$ be a place of $F$ . Then
$$
v_P (\C_q (\omega)) \geq 
\left\lfloor \frac{v_P (\omega)}{q} \right\rfloor
$$
and the above inequality holds even if $v_P (\omega)$ is negative.  
\end{cor}

\begin{proof}
Set $s := v_P (\omega)$ and let $b,r$ be such that $s = bq + r$ with $q > r \geq 0$. Clearly,
$b = \lfloor s/q \rfloor$. Let $z$ be a uniformising parameter at $P$ . There exists $\mu \in \Omega_{F/\F_q}$ with
$v_P (\mu) \geq 0$ such that $\omega = z^{bq} \mu$. Then, from Proposition \ref{glob}(\ref{lin}), we have $\C_q (\omega) = z^b \C_q (\mu)$
and, from Proposition \ref{local}(\ref{posit}), we have $v_P (\C_q (\mu)) \geq 0$. Thus, $v_P (\C_q (\omega)) \geq b$, which con-
cludes the proof.  
\end{proof}

\begin{cor}
Let $H$ be a (possibly non-positive) divisor on $X$ and $H_1$ be another divi-
sor such that $H \geq q H_1$ . Then, for all $\omega \in \Omega (H)$, we have $\C_q (\omega) \in \Omega (H_1)$.  
\end{cor}

\begin{proof}
It is a straightforward consequence of Corollary \ref{cor:added}.
\end{proof}

\subsection{The key vanishing lemma}
The following result is crucial in what follows. 

\begin{thm}\label{Vanish}
  Let $\omega \in \Omega_{F/\F_q}$, let $P$ be a place of $F$ and $s$ be a positive integer.
Assume that $\C_q(\omega)=\omega$ and $v_P (\omega)\geq sq-1$ for some positive integer $s$, then $v_P(\omega)\geq sq$.
\end{thm}

\begin{proof}
  Let $z$ be a uniformising parameter at $P$.
%  There exists a form $\mu$ regular at $P$ such that 
% $$
% \omega=a z^{sq-1}dz+z^{sq}\mu= a z^{sq} ( \frac{dz}{z} + \mu), 
% $$
% where $a$ is in the residue field $\F_q(P)$ of $P$.
The differential form $\omega$ is of the form
$\omega = z^{sq-1}\mu$, where $v_P (\mu)\geq 0$.
Set $\alpha :=\res_P (z^{-1}\mu) \in \F_q(P)$.
%If $\alpha = 0$, then $v_P (\mu)\geq 1$ and the result is proved. Thus, one can assume that $\alpha \neq 0$.
Since the residue field $\F_q (P)$ is perfect, $\alpha ^{1/q}$ is also an element of $\F_q (P)$. Let $a\in \O_{X,P}$ be a function such that $a \equiv \alpha^{1/q} \mod \mathfrak{m}_{X,P}$.
Then, $v_P(\mu-a^q dz)\geq 1$ and hence $\mu=a^q dz + z\eta$ with $v_P (\eta) \geq 0$.
Therefore, we have $\omega = z^{sq-1} (a^q dz + z\eta)$.
Applying $\C_q$ and using Propositions \ref{local} and \ref{glob}, we get
\begin{eqnarray*}
\C_q (\omega) & = & az^s \C_q \left( \frac{dz}{z} \right) + z^s \C_q (\eta)\\
              & = & az^{s-1}dz + z^s \C_q (\eta).
\end{eqnarray*}
Since $v_P (\eta)\geq 0$, from Proposition \ref{local}(\ref{posit}), we have $v_P(\C_q(\eta))\geq 0$.
In addition, since $s>0$ we have $s-1<sq-1$ and the assumption $\C_q (\omega)=\omega$, entails
$v_P (a) > 0$, that is $v_P (\omega)=v_P (z^{sq-1}(a^q dz+ z\eta))\geq sq$.
\end{proof}

% \begin{rem}
%   When the base field is $\F_2$, using Proposition \ref{glob}(\ref{log}), the above statement becomes: \emph{If a logarithmic differential over $\F_2$ vanishes at some place, then it vanishes with multiplicity at least $2$.}
% \end{rem}

\section{Codes defined using the Cartier operator}\label{secCar}

In this section, we introduce a new class of codes which turn out to be a natural geometric generalisation of classical Goppa codes.

\subsection{Motivation}\label{sec:motiv}
To motivate our construction, let us give an alternative proof of Theorem \ref{Thm:Egalite}. In some sense, it is a geometric version of the proof based on the error--locator polynomial \cite[Theorem 12.6]{slmc1}.

\begin{proof}[Proof of Theorem \ref{Thm:Egalite}]
Inclusion ``$\supseteq$" is elementary.
Conversely, using the notations of Example \ref{Goppa=diff}, we know that
$\Gamma_{q}(L, f^{q-1})=C_{\Omega}(D,(q-1)E-P)_{|\F_q},$
where $E$ is the divisor of the zeroes of $f$.
Let $c\in C_{\Omega}(D,(q-1)E-P)_{|\F_q}$. It is the image of a $1$--form $\omega \in \Omega ((q-1)E-P-D)$ by the map $\res_D$ introduced in Definition \ref{defn:CodeGeo}. The residues of $\omega$ at $P_1, \ldots , P_n$ are  in $\F_q$.
Since $\omega$ is regular everywhere but at the $P_i$'s and at $P$, then from the residue formula, $\res_{P}(\omega) \in \F_q$.

From Proposition \ref{local}(\ref{val-1}), the $1$--form $\C_q (\omega)$ has valuation $\geq -1$ at the $P_i$'s and at $P$. From Prop~\ref{local}(\ref{posit}) the form $\C_q (\omega)$ is regular out of the $P_i$'s and $P$. Finally, from Proposition~\ref{local}(\ref{resi}), the $1$--form $\C_q(\omega)$ has the same residues as $\omega$ at these points.
Thus,  $\C_q(\omega)-\omega$ has residues equal to $0$ at all the $P_i$'s and at $P$. Therefore, it is regular everywhere on $\P^1$ and hence is zero.
Consequently, $\C_q(\omega)=\omega$. In addition, since $f$ is squarefree, the divisor $E$ is reduced 
and, using Theorem \ref{Vanish}, we conclude that $\omega \in \Omega ((q-1)E-P-D)$ entails $\omega \in \Omega (qE-P-D)$ and hence $c\in C_{\Omega}(D, qE-P)_{|\F_q}=\Gamma_q (L,f^q)$.  
\end{proof}
 
If one tries to generalise these arguments to a higher genus curve, the proof fails since, nonzero regular differential forms exist. 
Therefore, the point of the following construction is to restrict to differential forms fixing $\C_q$.
%From this construction, we prove Theorem \ref{EgaliteGen}: a generalisation of Theorem \ref{Thm:Egalite} which entails Wirtz's Theorem \ref{Thm:Wirtz_2}.

\subsection{Context}\label{ContCodeCart}
In this section, $X$ is a curve of genus $g$ over $\F_{q^\ell}$ with $\ell \geq 1$. Let $P_1, \ldots , P_n$ a family of $\F_{q^\ell}$--rational points of $X$ and set $D:=P_1+\cdots +P_n$. 
Recall that the function field of $X$ is denoted by $F$ and its space of rational differential forms by $\Omega_{F/\F_{q^\ell}}$. 
Recall also that we denote by $\C_q$ and the map $\C^\ell$ where $q=p^\ell$ and $p$ is the characteristic.

\subsection{The codes} 

\begin{nota}
Let $\varphi$ be an endomorphism of a vector space $E$ and $A\subset E$, we denote by $A^{\varphi}$ the set of elements of $A$ fixed by $\varphi$, that is
$
A^\varphi:=A\cap \ker (\varphi-\textrm{Id}).
$
\end{nota}

\begin{defn}[The code $Car_q (D,G)$]
Let $G$ be a divisor on $X$ whose support avoids that of $D$. The code $Car_{q}(D,G)$ is a code over $\F_q$ defined as the image of the map.
$$
\res_D: \left\{
\begin{array}{ccc}
\Omega (G-D)^{\C_q} & \longrightarrow & \F_q^n \\
\omega & \longmapsto &
(\res_{P_1}(\omega), \ldots , \res_{P_n}(\omega))  
\end{array} \right. \cdot
$$
\end{defn}

  Even if $\Omega (G-D)$ is defined over $\F_{q^{\ell}}$, the code is actually defined over $\F_q$ because of Proposition \ref{local}(\ref{resi}).
This observation has in particular the following consequence.

\begin{prop}\label{prop:subf}
  The code $Car_q (D,G)$ is a subcode of $C_{\Omega}(D,G)_{|\F_q}$.
\end{prop}

In \S \ref{sec:choses}, we give explicit examples where $Car_q (D,G)$ is a proper subcode $C_{\Omega}(D, G)_{\F_q}$

The following theorem is a generalisation of Theorem \ref{Thm:Egalite}.

\begin{thm}\label{EgaliteGen}
Let $X, D$ be as in \S \ref{ContCodeCart} and $G$ be a divisor on $X$ whose support avoids that of $D$. Let $G_U$ be sum of the places $P$ such that $v_P(G) \geq 0$ and $v_P(G)\equiv q-1 \mod q$. Then,
$$
Car_q(D, G)=Car_q(D, G+G_U).
$$
\end{thm}

\begin{proof}
It is a straightforward consequence of Theorem \ref{Vanish}.
\end{proof}

\begin{cor}
  Let $G_0$ be a reduced positive divisor on $X$ and $E$ be another positive divisor. Assume that $G_0, E$ and $D$ have pairwise disjoint supports, then
$$
Car_q (D, (q-1)G_0-E)=Car_q (D, qG_0-E).
$$
\end{cor}

\begin{rem}
Compared to Wirtz's Theorem \ref{Thm:Wirtz_2}, Theorem \ref{EgaliteGen} holds for all divisor $G$ without any condition on its degree, while there exist divisors $G$ such that $C_{\Omega}(D,G+G_U)_{|\F_q} \varsubsetneq C_{\Omega}(D, G)_{|\F_q}$ (see \S \ref{sec:trucs_et_machins}). For this reason, our new construction seems to be a more natural geometric generalisation of classical Goppa codes than subfield subcodes of AG codes.
\end{rem}

\section{Comparing the two constructions}\label{sec:compare}

\begin{thm}\label{thm:codim}
Let $X, D, G$ be as in Definition \ref{defn:CodeGeo} and $G_1$ be a divisor such that $G\geq qG_1$ and $G\geq G_1$. Then,
$$
\dim_{\F_q}  \raisebox{.5ex}{\ensuremath{C_{\Omega} (D, G)_{|\F_q}}} \bign/
\! \raisebox{-.65ex}{\ensuremath{Car_q (D, G)}} \leq \ell h^1 (G_1).
$$ 
In particular, 
$h^1(G_1)=0 \ \Longrightarrow \  C_{\Omega} (D, G)_{|\F_q} = Car_q (D, G).$
\end{thm}

\begin{proof}
First, notice that $G-D \geq q(G_1-D)$ and hence, from Corollary \ref{cor:added}, for all $\omega \in \Omega (G-D)$, we have $\C_q (\omega) \in \Omega (G_1-D)$. Corollary \ref{cor:added} also asserts that if $\omega \in \Omega (G)$, then $\C_q (\omega)\in \Omega (G_1)$.

Now, consider the following morphism of exact sequences:
$$
\xymatrix{
\relax 0 \ar[r] & \Omega (G) \ar[r] \ar[d]^{\C_q -Id} & \Omega (G-D) \ar[r]^{\res_D} \ar[d]^{\C_q - Id} & C_{\Omega}(D,G) \ar[r] \ar[d]^{\phi^{-1}-Id} & 0\\
0 \ar[r] & \Omega(G_1) \ar[r] & \Omega(G_1-D) \ar[r]^{\res_D} & C_{\Omega}(D, G_1) \ar[r] & 0
},
$$
where $\res_D$ is the map introduced in Definition \ref{defn:CodeGeo} and $\phi: \F_{q^{\ell}}^n \rightarrow \F_{q^{\ell}}^n$ is the coordinate-wise Frobenius map $(c_1, \ldots, c_n) \mapsto (c_1^q, \ldots , c_n^q)$. The left-hand square is clearly commutative and the commutativity of the right--hand one is an easy consequence of Proposition \ref{local}(\ref{resi}).

From the Snake Lemma, we get the exact sequence
$$
0\ \longrightarrow \Omega (G)^{\C_q} \ \longrightarrow \ \Omega (G-D)^{\C_q}\ \longrightarrow \ C_{\Omega}(D, G)_{|\F_q} \longrightarrow 
 \raisebox{.5ex}{\ensuremath{ \Omega (G_1) }} \bign/
\! \raisebox{-.65ex}{\ensuremath{ (\C_q - Id)\Omega(G) }}\ ,
$$
which naturally entails
$$
0 \ \longrightarrow \ Car_q (D,G) \ \longrightarrow \ C_{\Omega} (D, G)_{|\F_q} \ \longrightarrow \ 
 \raisebox{.5ex}{\ensuremath{ \Omega (G_1) }} \bign/
\! \raisebox{-.65ex}{\ensuremath{ (\C_q - Id)\Omega(G) }}.
$$
This yields the result.
\end{proof}

% \begin{rem}
%   It is worth noting that the cases where $Car_q (D,G)$ and $C_{\Omega}(D,G )_{|\F_q}$ may differ are the most interesting ones. Indeed, the codes may differ when the degree of $G$ is small, hence when the dimension of $C_{\Omega} (D,G)$ is large. Since the subfield subcode operation makes dramatically collapse the dimension, to get $\F_q$--codes with reasonable dimensions one needs to choose a divisor $G$ of low degree. 
% \end{rem}

\begin{rem}\label{rem:alter_proof}
Wirtz's Theorem \ref{Thm:Wirtz_2} is a straightforward consequence of the previous results. Indeed, the condition $\deg (G_1)>2g-2$ in the statement asserts that $h^1 (G_1)=0$ and hence, from Theorem \ref{thm:codim}, we have $Car_q(D,G)=C_{\Omega}(D, G)_{|\F_q}$. Applying Theorem \ref{EgaliteGen} to the Cartier Codes, we get Wirtz's result. 
\end{rem}

\section{Parameters of Cartier codes}\label{sec:param}
In this section, we give lower bounds for the parameters of Cartier codes and discuss in particular their dimension by giving two distinct lower bounds. The first one derives from Stichtenoth's Theorem \ref{thm:dim_sticht} together with Theorem \ref{thm:codim}. The second one is direct and is proved without using subfield subcodes of AG codes.

Since we always have $Car_q (D,G) \subset C_{\Omega}(D, G)_{|\F_q}$ (Proposition \ref{prop:subf}), this second lower bound holds for subfield subcodes of AG codes and improves in some situations Stichtenoth's bound.

\subsection{Minimum distance}\label{sec:mindist}
In what follows, we denote by $d(C)$ the minimum distance of a code $C$.
A natural lower bound for the minimum distance of the codes $Car_q (D,G)$ is given by the inclusion $Car_q (D, G) \subset C_{\Omega} (D,G)$ and the Goppa designed distance. Hence
$$
d(Car_q (D,G)) \geq d(C_{\Omega}(D,G))\geq \deg (G)+2-2g.
$$
Moreover, Theorem \ref{EgaliteGen} improves the minimum distance in some situation. Namely, 
in the context of Theorem \ref{EgaliteGen}, we have
$$
d(Car_q(D,G))\geq \deg(G+G_U)+2-2g.
$$

\subsection{First lower bound for the dimension}

\begin{thm}\label{dimC}
Let $X, D, G$ be as in Definition \ref{defn:CodeGeo}, let $G$ be a divisor on $X$ and $G_1$ be a divisor such that $G\geq qG_1$ and $G \geq G_1$.
Then
\begin{equation}\label{eq:ineqdim1}
\dim_{\F_q} (Car_q (D,G)) \geq \left\{
  \begin{array}{lll}
    n- 1 - \ell (h^0(G)-h^0(G_1)+h^1 (G_1)) & {\rm if} & G\geq 0\\
    n- \ell (h^0(G)-h^0(G_1)+h^1 (G_1)) & {\rm if} & G \ngeq 0
  \end{array}
\right. .
\end{equation}
Moreover, if $h^1 (G) = 0$, then
\begin{equation}\label{eq:ineqdim2}
\dim_{\F_q} (Car_q (D,G)) \geq \left\{
  \begin{array}{lll}
    n- 1 - \ell \deg (G-G_1) & {\rm if} & G\geq 0\\
    n- \ell \deg (G-G_1) & {\rm if} & G \ngeq 0
  \end{array}
\right. .
\end{equation}
\end{thm}

\begin{proof}
  Inequalities (\ref{eq:ineqdim1}) are a straightforward consequence of Theorems \ref{thm:dim_sticht}  and \ref{thm:codim}. Inequalities (\ref{eq:ineqdim2}) are consequences of (\ref{eq:ineqdim1}) together with the Riemann--Roch Theorem.
\end{proof}

\begin{cor}
  \label{cor:easy}
Let $G_0$ be a reduced positive divisor on $X$ whose support is disjoint from that of $D$ and such that $h^1 (qG_0)=0$. Then the code $Car_q (D, qG_0)$ has parameters satisfying
$$
\begin{array}{ccc}
  k & \geq & n-1-\ell (q-1)\deg (G_0)\\
  d & \geq & q\deg (G_0)+2-2g
\end{array}
$$
\end{cor}

\begin{proof}
  From Theorem \ref{EgaliteGen}, we have $Car_q (D, qG_0)=Car_q (D, (q-1)G_0)$. Then, apply Theorem \ref{dimC} to $Car_q (D, (q-1)G_0)$ to get the dimension and apply the bounds of \S \ref{sec:mindist} to $Car_q (D,qG_0)$ to get the minimum distance.
\end{proof}

\subsection{Second lower bound for the dimension}

The lower bounds for the dimension of Cartier codes of Theorem \ref{dimC} come from Stichtenoth's estimates for the dimension of subfield subcodes. Here, we state a bound which can be proven directly without using subfield subcodes.

\begin{thm}\label{thm:improve_dim}
Let $X,D, G$ be as in Definition \ref{defn:CodeGeo}. Let $G^+, G^-$ be the two positive divisors with disjoint supports such that $G=G^+-G^-$.
Assume that $G^-$ is reduced and that $G^+, G^-$ and $D$ have pairwise disjoint supports.
Let $s_{G^-}$ be the number of places supporting $G^-$.
Then, 
$$
\dim_{\F_q} Car_q (D,G) \geq n-1+s_{G^-} -\ell \deg (G^+) - h^1 (G).  
$$
\end{thm}

\begin{proof}
\noindent {\it Step 1.} The code $Car_q (D, G)$ is the image of the $\F_q$--linear map
$$
\phi:\left\{
  \begin{array}{ccc}
    \Omega(G-D)^{\C_q} & \longrightarrow & \F_q^n\\
    \omega & \longmapsto & (\res_{P_1}(\omega), \ldots , \res_{P_n}(\omega))
  \end{array}
\right. .
$$
The kernel of the map is $\Omega(G)^{\C_q}$ whose $\F_q$--dimension is bounded above by $h^1 (G)$. Indeed, notice that $\Omega (G)\supseteq \Omega (G)^{\C_q} \otimes_{\F_q} \F_{q^{\ell}}$.

Now, let us bound below the dimension of $\Omega (G-D)^{\C_q}$. 

\medbreak

\noindent {\it Step 2.} Obviously, we have $G-D \geq  q(-G^- -D)$. Thus, from Corollary \ref{cor:added}, the following map is well--defined.
\begin{equation}\label{eq:kernel}
\C_q -Id :    \Omega (G-D) \longrightarrow \Omega(-G^--D)
\end{equation}
and the $\F_q$--space, $\Omega (G-D)^{\C_q}$ is its kernel.
From now on, denote by $V$ the image of the above map.
Notice that $V$ is an $\F_q$--subspace of $\Omega (-G^- -D)$ and not an
$\F_{q^{\ell}}$--subspace in general.
We claim that the map (\ref{eq:kernel}) is not surjective and will construct a
proper $\F_q$--subspace of $\Omega (-G^- -D)$ containing $V$.

\medbreak

\noindent {\it Step 3.} Recall that, given a place $P$ of $X$, we denote by $\F_{q^{\ell}} (P)$ the corresponding residue field.
Let $Q_1 , \ldots, Q_{s_{G^-}}$ be the places supporting $G^-$.
Now, consider the $n-1 + s_{G^-}$ following $\F_q$--linear forms on $\Omega (-G^- -D)$:
\begin{equation}
  \label{eq:trace}
 \psi_P : \omega \rightarrow \textrm{Tr}_{\F_{q^{\ell}}(P)/\F_q} (\res_P (\omega)) \quad \textrm{for}\ P\in \{Q_1, \ldots , Q_{s_{G^-}}, P_1, \ldots , P_{n-1}\}.
\end{equation}

Elements of $V$ are of the form $\C_q (\omega) - \omega$ and hence,
from Proposition \ref{local}(\ref{resi}), the traces of their residues
are always zero.
Therefore, the above--described maps $\psi_P$ vanish on $V$.

In addition, the $\F_q$--linear forms $\psi_P$ described in (\ref{eq:trace})
are independent on $\Omega(-G^- -D)$.
Indeed, for all place $P \in \{Q_1 , \ldots ,Q_{s_{G^-}} , P_1 , \ldots , P_{n-1} \}$, Riemann--Roch Theorem asserts that $\Omega(0) \varsubsetneq \Omega(-P - P_n)$.
Then, choose a form $\omega_P \in \Omega (-P-P_n) \setminus \Omega (0)$.
From the residue formula, $\res_P (\omega_P) \neq 0$ for all $P$.
The forms $\omega_{Q_1}, \ldots , \omega_{Q_{s_{G^-}}}, \omega_{P_1}, \ldots ,
\omega_{P_{n-1}}$ are elements of $\Omega (-G^- -D)$ and provide a dual basis
for the $\F_q$--linear forms described in (\ref{eq:trace}),
which yields the independence of these maps.

Finally, $V$ is contained in the intersection of the kernels of
$n - 1 + s_{G^-}$ independent $\F_q$--linear forms on $\Omega (-G^- -D)$ and hence its codimension in this space is at least $n-1+s_{G^-}$.

\medbreak

\noindent {\it Step 4.} From Riemann--Roch Theorem and the previous step, the dimension of the image $V$ of the map (\ref{eq:kernel}) satisfies
\begin{equation}
  \label{eq:dimV}
  \dim_{\F_q} (V) \leq \ell (g+n-1+\deg G^-) - (n-1+s_{G^-}).
\end{equation}
On the other hand, from Riemann--Roch Theorem, we also have
\begin{equation}
  \label{eq:dimOm}
  \dim_{F_q} \Omega (G-D) \geq \ell (n-\deg G + g -1).
\end{equation}
Combining (\ref{eq:dimV}), (\ref{eq:dimOm}) and Step 1, we get the result.
\end{proof}

\begin{rem}
  For $G_1\geq 0$ reduced, $G=qG_1$ and $h^1(G)=0$, then Theorem \ref{thm:improve_dim} gives the same bound as Theorem \ref{dimC}.
\end{rem}

\subsection{When Cartier Codes improve the bounds on the dimension of subfield subcodes of Algebraic Geometry codes}

\begin{cor}\label{cor:impr}
  Let $G_0, G^-$ be two positive reduced divisors on $X$ such that $G_0, G^-$ and $D$ have pairwise disjoint supports and $h^1(G_0-G^-)=0$. Set $G:=qG_0-G^-$ and let $s_{G^-}$ be as in Theorem \ref{thm:improve_dim}. Then, 
$$
\dim_{\F_q} C_{\Omega} (D, qG_0-G^-)_{|\F_q} \geq n-1+s_{G^-} - \ell (q-1)\deg (G_0),
$$
which improves Theorem \ref{thm:dim_sticht} as soon as $s_{G^-}>1$.
\end{cor}

\begin{proof}
Set $G_1:=G_0-G^-$. Clearly, we have $G\geq qG_1$ and $G\geq G_1$.
  From Theorem \ref{thm:codim}, the assumption $h^1(G_0-G^-)=h^1(G_1)=0$ entails $Car_q (D,G)= C_{\Omega} (D,G)_{|\F_q}$. Then, from Theorem \ref{EgaliteGen}, we have $Car_q (D, (q-1)G_0-G^-)=Car_q (D, qG_0-G^-)$.
We conclude using Theorem \ref{thm:improve_dim}.
\end{proof}

As a comparison, Stichtenoth's Theorem \ref{thm:dim_sticht}, yields $n-\ell (q-1)\deg (G_0)$ as a lower bound for the dimension instead of $n-1+s_{G^-}-\ell (q-1)\deg (G_0)$.

\subsection{Infinite families of codes}\label{asymp}
Given an infinite family of codes $(C_i)_{i\in \N}$ with parameters $[n_i, k_i, d_i]$ with $n_i \rightarrow +\infty$, recall that we denote by $R$ and $\delta$ the asymptotic parameters of the family defined as:
$$
R:=\limsup_{i\rightarrow +\infty} \frac{k_i}{n_i} \qquad \delta :=\limsup_{i\rightarrow +\infty} \frac{d_i}{n_i} \cdot
$$
In \cite{KatsmanTsfasman}, the authors discuss the asymptotic performances of subfield subcodes of AG codes. For all even $\ell$, i.e. when $q^{\ell}$ is a square, they prove the existence of infinite families of such codes whose asymptotic parameters satisfy
\begin{equation}\label{eq:asymp}
R \geq 1-\frac{2(q-1)\ell}{q(q^{\ell /2} -1)}-\frac{(q-1)\ell}{q}\delta \qquad {\rm for}\qquad \frac{q-2}{q^{\ell/2}-1} \leq \delta \leq \frac{q}{m(q-1)} - \frac{2}{q^{\ell/2}-1}\cdot
\end{equation}
They prove in particular that such codes reach the Gilbert Varshamov bound for $\delta \sim 0$.

Now, let $G_0$ be a reduced positive divisor on $X$ such that $q\deg (G_0)>2g-2$ and consider a family of codes of the form $Car_q (D, qG_0)=Car_q (D, (q-1)G_0)$. Then, from Corollary \ref{cor:easy}, their parameters satisfy
$$
k \geq n-1-\frac{2(q-1)\ell}{q}g - \frac{(q-1)\ell}{q}d.
$$
If $\ell$ is even and hence if $q^{\ell}$ is a square, then the Drinfel'd Vl\u{a}du\c{t} Theorem (\cite[Theorem 3.2.3]{TVN}) asserts the existence of a family of Cartier codes whose parameters $R,\delta$ satisfy exactly the left--hand inequality of (\ref{eq:asymp}).
 In addition, the existence of a reduced positive divisor $G_0$ with $q\deg (G_0)>2g-2$ is asserted whenever the conditions on $\delta$ of (\ref{eq:asymp}) hold. See \cite{KatsmanTsfasman} for further details on the construction of such a divisor.

As a conclusion, Cartier codes and Subfield subcodes of AG codes have similar asymptotic performances.

\section{An Example}\label{sec:examples}

Computations are made using {\sc Magma} \cite{Magma}. A program generating Cartier Codes is available on the author's webpage.

\medbreak

Consider the Klein Quartic of equation $x^3 y+y^3 z +xz^3$ over $\F_8$.
This curve has genus $3$. It has $24$ $\F_8$--points including $P_1:=(0:1:0)$, $P_2:=(0:0:1)$ and $P_3:=(1:0:0)$. Denote by $Q_1, \ldots, Q_{21}$ the other rational points.  We also introduce $3$ places of degree $2$. Let $w$ be a primitive element of $\F_8/\F_2$ with minimal polynomial $T^3+T+1$ and $R_1, R_2, R_3$ be the three places of degree $2$ defined by the ideals: $\langle y^2 + w^5 y z + w^3 z^2,     x + w^3 y + wz \rangle,\ \langle   y^2 + w^3 y z + w^6 z^2, x + w^6 y + w^2 z  \rangle, \ \langle  y^2 + yz + z^2, x + y + z \rangle$.

We will construct codes over $\F_2$, i.e. $\ell=3$.
Set
$$
\begin{array}{ccc}
 D & := & Q_1+\cdots +Q_{21}\\
 G_0 & := & R_1+R_2+R_3  \\
 G^- & := & P_1+P_2+P_3.
\end{array}
$$
Computer-aided calculations give
\begin{equation}
  \label{eq:H1}
  h^1(G_0-G^-)=0 \qquad h^1(-G^-)=5
\end{equation}
and the following triple of parameters:
\begin{equation}
  \label{eq:computed_param}
\begin{array}{crccr}
  Car_2 (D, G_0-G^-): & [21, 6, 8]_2 & \qquad & C_{\Omega} (D, G_0-G^-)_{|\F_2}: & [21, 18, 2]_2 \\
   Car_2 (D, 2G_0-G^-): & [21, 6, 8]_2 & \qquad & C_{\Omega} (D, 2G_0-G^-)_{|\F_2}: & [21, 6, 8]_2
\end{array}
\end{equation}
This example illustrates several results presented before.

\subsection{Illustration of Theorem \ref{EgaliteGen}}
  \label{sec:trucs_et_machins}
Theorem \ref{EgaliteGen} asserts that 
$$
Car_2 (D,G_0-G^-)= Car_2 (D, 2G_0-G^-).
$$
It is confirmed in (\ref{eq:computed_param}). Indeed, the inclusion $Car_q (D, G_0-G^-) \supseteq Car_q (D, 2G_0-G^-)$ is obvious and both codes have dimension $6$. On the other hand, computing the parameters of the codes $C_{\Omega}(D,G_0-G^-)_{|\F_2}$ and $C_{\Omega}(D, 2G_0-G^-)_{\F_2}$ we observe that the codes are distinct and have respective parameters $[21, 18, 2]_2$ and $[21, 6, 8]_2$. Here, Theorem \ref{EgaliteGen} holds while
$$
C_{\Omega} (2G_0-G^-)_{|\F_2} \varsubsetneq C_{\Omega}(D,G_0-G^-)_{|\F_2}.
$$

\subsection{Illustration of Theorem \ref{thm:codim}}
  \label{sec:choses}
Set $G_1:=G_0-G^-$, we have $2G_0-G^- \geq 2G_1$ and $2G_0-G^- \geq G_1$. From (\ref{eq:H1}), $G_1$ is non special. Thus, Theorem \ref{thm:codim} asserts that $Car_2 (D, 2G_0-G^-)=C_{\Omega}(D, 2G_0-G^-)_{|\F_2}$, which is confirmed by the experience: they both have dimension $6$ and, from Proposition \ref{prop:subf}, the Cartier code is contained in the second one.

On the other hand, if we compare $Car_2 (D, G_0 -G^-)$ and $C_{\Omega}(D, G_0-G^-)_{\F_2}$ one can apply Theorem \ref{thm:codim} with $G_1=-G^-$. Then, from (\ref{eq:H1}), we have $h^1(-G^-)= 5$. Since $\ell=3$, the difference between the codimensions of the codes is at most $15$. The actual codimension is $12$. It is in particular an example of non equality $$Car_2 (D, G_0-G^-) \varsubsetneq C_{\Omega} (D, G_0-G^-)_{|\F_2}.$$

\subsection{Illustration of Theorem \ref{thm:improve_dim}}
In this example, Stichtenoth's Theorem \ref{thm:dim_sticht} asserts that
$\dim C_{\Omega}(D, G)_{|\F_2}\geq 3$, while Theorem \ref{thm:improve_dim} (or Corollary \ref{cor:impr}) asserts that this dimension is at least $5$ (we have $s_{G^-}=3$). As said before, the actual dimension is $6$.

\section*{Conclusion}
We give a new construction of codes which seems to be the most natural algebraic geometric generalisation of classical Goppa codes. In particular these codes satisfy equalities which are very similar to the relation $\Gamma (L, f^{q-1})=\Gamma (L, f^q)$ satisfied by Goppa codes.

In addition, we are able to bound below their parameters.
Our bounds on the dimension are obtained by two different manners, first by bounding above the codimension of the Cartier code as a subcode of the corresponding subfield subcode (Theorems~\ref{thm:codim} and \ref{dimC}).
Second, by a direct proof without using the known estimates on the  dimension of subfield subcodes. 
This second bound has a nice application to subfield subcodes of AG codes, since it improves in some situations Stichtenoth's bound for the dimension (Theorem~\ref{thm:dim_sticht}).

\section*{Acknowledgements}
The author wishes to thank Daniel Augot who encouraged him to work on this topic. He also expresses his gratitude to Niels Borne for many inspiring discussions.
A part of this work has been done when the author was a Post Doc researcher supported by the French ANR Defis program under contract
ANR-08-EMER-003 (COCQ project).
Special thanks to the anonymous referee for his/her efficiency and for the relevance of his/her comments. 

\bibliographystyle{abbrv}
\bibliography{biblio}
\end{document}